\documentclass{article}
\usepackage{amsmath}
\usepackage{amsthm}
\usepackage{amssymb}
\usepackage{mathtools}
\usepackage{empheq}
\newtheorem{theorem}{Theorem}

\usepackage{graphicx}
\usepackage{subcaption}
\begin{document}
\title{\rightline{\normalsize{Preprint. Under Review for the CCP2023 Proceedings.}}
\vskip2\baselineskip
Elliptical Pursuit and Evasion \\
\normalsize{Extended Version}}

\author{Sota Yoshihara$^{1}$ \\
\small \ $^{1}$Graduate School of Mathematics, Nagoya University, Japan\\ 
\small email:~sota.yoshihara.e6@math.nagoya-u.ac.jp\\
\small WWW homepage:~https://scholar.google.co.jp/citations?user=Z0eb2QcAAAAJ\&hl=ja
}
\date{}
\maketitle
 \begin{abstract}
Many studies on one-on-one pursuit-evasion problems have shown that formulas about the pursuer's trajectory can be solved by supposing three conditions. First, the evader follows specific figures. Second, the pursuer's velocity vector always points toward the evader's position. Third, the ratio of their respective speed remains constant. However, previous studies often assumed that the evader moves at a steady speed. This study aims to investigate how changes in the evader’s speed affect the pursuer’s trajectory. We hypothesized that the pursuer's trajectory would remain unchanged. First, the pursuer's trajectories were obtained from three scenarios where the evader orbits an ellipse with different speeds and angular velocities. These trajectories coincided. Second, changes in the evader’s speed correspond to changes in the evader’s trajectory parameters. Replacing the evader's parameter is proven to be replacing the pursuer's parameter. It is shown that replacing the evader’s parameter is equivalent to replacing the pursuer’s parameter. Consequently, the shape of the pursuer’s trajectory is unaffected by the evader’s speed; only the speed ratio matters in the game.

This version includes additional sections on the dynamical system that were not present in the original version.  If the evader’s speed is always one, a dynamical system can be derived from the three conditions of pursuit and evasion. When the evader orbits a circle, this dynamical system is autonomous and has an asymptotically stable equilibrium point. However, when the evader orbits an ellipse, the dynamical system becomes non-autonomous, and the solution trajectory converges to a closed curve. Additionally, we present a second-order nonlinear differential equation describing the angular difference between the velocity vectors of both players.
\end{abstract}

\section{Introduction}
This study explores the classical one-on-one pursuit and evasion problem in a two-dimensional plane. We focus on analyzing the pursuer’s trajectory under three conditions:
\begin{enumerate}
\item The evader's movement is unaffected by the pursuer.
\item The pursuer's velocity vector constantly points toward the evader's position.
\item The ratio of the pursuer's speed to the evader's speed remains constant.
\end{enumerate}
Differential equations that formalize this problem are given in Section 2. 

Notably, when the evader follows a circular trajectory, this problem is referred to as Hathaway's dog and duck pursuit problem. Moreover, when the pursuer's speed is $n~(n<1)$ times slower than the evader, the pursuer's trajectory converges to a circle reduced by a factor of $n$. \cite{Nahin, Ohira} We investigated the pursuer’s trajectory when the evader orbits an ellipse, as discussed in \cite{Sota}, and found that this phenomenon does not occur (see Fig.\ref{CircularEllipceDifference}).

In \cite{Sota}, the evader's speed and angular velocity were not constant. We hypothesize that changing the ellipse parameters will alter the pursuer's trajectory. This hypothesis is tested in Section 3, where we derive two alternative sets of ellipse parameters with constant speed and constant angular velocity. We calculate three numerical solutions about the pursuer's trajectory corresponding to these three different scenarios where the evader orbits an ellipse. The resulting three pursuer trajectories were consistent, disproving the hypothesis.

Experimental data show that for parameter values that result in the same evader coordinates across different parameterizations, the pursuer coordinates are also consistent. This indicates that replacing the evader’s parameters is equivalent to replacing the pursuer’s parameters, leaving the shape of the pursuer’s trajectory unchanged. In Section 4, we mathematically prove this supposition.

Sections 5-8 are additional sections to the original version for the CCP2023 proceedings. In these sections, we assume the evader’s velocity is 1. According to the theorem presented in Section 4, this assumption does not affect the shape of the pursuer’s trajectory. In Section 5, we derive the dynamical system using the angular difference between the velocity vectors of the two players and the distance between them as variables. Section 6 explores the properties of the dynamical system when the evader orbits a circle. In this scenario, the dynamical system is autonomous and has an asymptotically stable equilibrium point. In Section 7, we investigate the case where the evader orbits an ellipse. Here, the dynamical system is non-autonomous, and thus, it lacks an equilibrium point. Additionally, the solution trajectory converges to a closed curve. In Section 8, we propose a method to reduce the variables in the dynamical system from two to one. This leads to a second-order nonlinear differential equation describing the angular difference between the velocity vectors of the two players.

 \begin{figure}[htbp]
 \begin{subfigure}[b]{0.46\linewidth}
  \centering
  \includegraphics[keepaspectratio, scale=0.45]{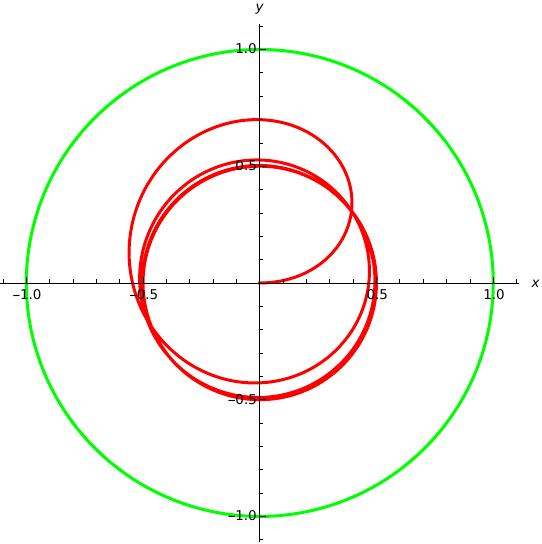}
  \caption{Red: Pursuer, Green: Evader.}
  \end{subfigure}
 \begin{subfigure}[b]{0.46\linewidth}
  \centering
  \includegraphics[keepaspectratio, scale=0.45]{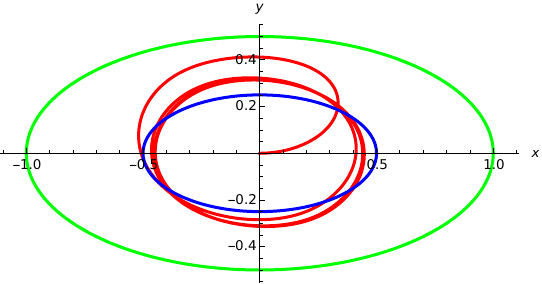}
  \caption{Red: Pursuer,  Green: Evader, Blue: An ellipse reduced by a factor of $n$.}
  \end{subfigure}
 \caption{Difference between circular and elliptical chase and escape. Here, $n$ denotes the ratio of Pursuer to Evader speed. We set $n=0.5$. }
 \label{CircularEllipceDifference}
\end{figure}
\section{Preliminary}
We formulate the problem of pursuit and evasion based on \cite{Barton} as a differential equation problem. Let $\mathbf{E}(t)$ and $\mathbf{P}(t)$ denote the evader and the pursuer position at time $t$, respectively. According to the first condition, $\mathbf{P}$ is given. The second condition can be expressed as follows:
\begin{align}
\lambda \dot{\mathbf{P}}&=\mathbf{E}-\mathbf{P}  \qquad \lambda(t)\ge 0,\label{eq-muki}
\end{align}
 where $\lambda$ represents the ratio of the distance between the pursuer and the evader to the pursuer’s speed. The third condition can be formulated as follows:
\begin{align}
|\dot{\mathbf{P}}|&=n|\dot{\mathbf{E}}|, \label{eq-speed}
\end{align}
where $n>0$ denotes the ratio of the pursuer to the evader speed. By putting $\mathbf{E}=(X(t),Y(t))$ and $\mathbf{P}=(x(t),y(t))$, the component equations for the vectors are as follows:
\begin{empheq}[left={\empheqlbrace}]{align}
&X=x+\lambda\dot{x}, \label{barton-tsuiseki1}\\
&Y=y+\lambda\dot{y}. \label{barton-tsuiseki2}\\
&n^2\left(\dot{X}^2+\dot{Y}^2\right)=\dot{x}^2+\dot{y}^2, \label{barton-tsuiseki3}
\end{empheq}
The problem of pursuit and evasion problem then reduces to solving for $x(t), y(t)$ and $\lambda (t)$ that satisfy equations \eqref{barton-tsuiseki1}- \eqref{barton-tsuiseki3} for given $X(t), Y(t)$ and $n$. These equations are challenging to solve numerically, but they can be transformed into a more solvable form. From equation \eqref{eq-muki} and \eqref{eq-speed}, 
\begin{align}
    \lambda=\frac{|\mathbf{E}-\mathbf{P}|}{n|\dot{\mathbf{P}}}=\frac{\sqrt{\dot{X(t)}^2+\dot{Y(t)}^2}}{\sqrt{(X(t)-x(t)^2+(Y(t)-y(t))^2}} 
\end{align}
Substituting this into the component equations for $X$ and $Y$ (\eqref{barton-tsuiseki1} and \eqref{barton-tsuiseki2}, respectively), we obtain the following two equations:
\begin{empheq}[left={\empheqlbrace}]{align}
\dot{x}&=\frac{n(X(t)-x(t))\sqrt{\dot{X(t)}^2+\dot{Y(t)}^2}}{\sqrt{(X(t)-x(t)^2+(Y(t)-y(t))^2}},\label{eq-x}\\
\dot{y}&=\frac{n(Y(t)-y(t))\sqrt{\dot{X(t)}^2+\dot{Y(t)}^2}}{\sqrt{(X(t)-x(t)^2+(Y(t)-y(t))^2}}.\label{eq-y}
\end{empheq}
Thus, the pursuit and evasion problem reduces to solving for $x(t)$ and $y(t)$ that satisfy equations \eqref{eq-x} and \eqref{eq-y} for given $X(t)$, $Y(t)$, and $n$. This method can calculate the solution only up until just before the pursuer catches the evader, as the denominators in these two equations approach zero at that point.

\section{Deriving Ellipse Parameters}
When the evader orbits an ellipse, $X$ and $Y$ satisfy the following equation:
\begin{align}
\frac{X^2}{a^2}+\frac{Y^2}{b^2}=1,\quad a>0, b>0.\label{EllipseDef}
\end{align}
In this paper, the evader starts from $(a, 0)$ and rotates counterclockwise. These are standard ellipse parameters.
\begin{align}
X_1(t)=a\cos{t},\quad Y_1(t)=b\sin{t}.
\label{param-1}
\end{align}
If $a=b$, the evader's speed and angular velocity are constant, but if $a\neq b$ neither is constant. In Fig.\ref{CircularEllipceDifference}(A), $a=1.0$, $b=1.0$, while in Fig.\ref{CircularEllipceDifference}(A), $a=1.0$, $b=0.5$. Therefore, we hypothesize that the lack of constant speed and angular velocity is the cause of the difference in the pursuer’s trajectory. We derive two ellipse parameterizations where either the speed or the angular velocity is constant. Note that no parameterization can ensure both remain constant simultaneously.
\subsection{Constant Angular Velocity}
We derive a paremeter $(X_2(t), Y_2(t))$ where the angular velocity is constant and equal to $1$. Given that the angular velocity is $1$, the argument at the point $(X_2(t), Y_2(t))$ is $t$. Therefore, we have: 
\begin{align}
X_2(t)=r(t)\cos{t},\quad Y_2(t)=r(t)\sin{t},
\label{3.1-1}
\end{align}
where $r(t)>0$. Substituting these into the ellipse equation\eqref{EllipseDef}:
\begin{align}
r(t)&=\frac{ab}{\sqrt{a^2\sin^2(t)+b^2\cos^2(t)}}.
\end{align}
Thus, the parameterization becomes:
\begin{align}
X_2(t)=\frac{ab\cos{t}}{\sqrt{a^2\sin^2{t}+b^2\cos^2{t}}},\quad
Y_2(t)=\frac{ab\sin(t)}{\sqrt{a^2\sin^2(t)+b^2\cos^2(t)}}.\label{param-2}
\end{align}
\subsection{Constant Speed}
Next, we derive a parameter  $(X_3(t), Y_3(t))$ where the speed is constant and equal to $1$. Directly determining $X_3(t)$ and $Y_3(t)$ is difficult. Instead, we parameterize the velocity vector as:
\begin{align}
\dot{X_3}=\cos{\varphi(t)},\quad \dot{Y_3}=\sin{\varphi(t)}.\label{3.2-1}
\end{align}
Differentiating the ellipse equation \eqref{EllipseDef} with respect to time $t$ and dividing by $2$, 
\begin{align}
\frac{X_3}{a}\frac{\dot{X_3}}{a}+\frac{Y_3}{b}\frac{\dot{Y_3}}{b}=0.
\end{align}
Hence, $(X_3/a,Y_3/b)$ and $(\dot{X_3}/a,\dot{Y_3}/b)$ are orthogonal. Since, $(X_3/a,Y_3/b)$ has a length of $1$ and parallel to the vector obtained by rotating $(\dot{X_3}/a,\dot{Y_3}/b)$ by $-\pi/2$. Hence, $(X_3/a,Y_3/b)$ is the vector which $(\dot{Y_3}/b,-\dot{X_3}/a)=(\sin{\varphi}/b,-\cos{\varphi}/a)$ is normalized. Then we find:
\begin{align}
X_3(\varphi)=\frac{a^2\sin{\varphi}}{\sqrt{a^2\sin^2{\varphi}+b^2\cos^2{\varphi}}},\quad
Y_3(\varphi)=-\dfrac{b^2\cos{\varphi}}{\sqrt{a^2\sin^2{\varphi}+b^2\cos^2{\varphi}}}\label{param-3}
\end{align}
To obtain the pursuer's trajectory, we convert $t$ in \eqref{eq-x} and \eqref{eq-y} to $\varphi$. The pursuer’s trajectory is determined by $x(\varphi)$ and $y(\varphi)$, satisfying the following two equations:
\begin{align}
\dfrac{dX_3(\varphi)}{d\varphi}&=n\dfrac{X_3(\varphi)-x(\varphi)}{\sqrt{(X_3(\varphi)-x(\varphi))^2+(Y_3(\varphi)-y(\varphi))^2}\dot{\varphi}}\label{eq-x-phi}\\
\dfrac{dY_3(\varphi)}{d\varphi}&=n\dfrac{Y_3(\varphi)-y(\varphi)}{\sqrt{(X_3(\varphi)-x(\varphi))^2+(Y_3(\varphi)-y(\varphi))^2}\dot{\varphi}}\label{eq-y-phi}
\end{align}
We also need to determine $\dot{\varphi}$. From equation \eqref{param-3}:
\begin{align}
X_3(\varphi)=-Y_3(\varphi)\dfrac{a^2}{b^2}\tan{\varphi}.
\end{align}
Differentiating with respect to time $t$ and substituting \eqref{3.2-1} and \eqref{param-3}, we obtain:
\begin{align}
\cos{\varphi}=-\sin{\varphi}\frac{a^2}{b^2}\tan{\varphi}+\frac{a^2\dot{\varphi}}{\cos{\varphi}\sqrt{a^2\sin^2{\varphi}+b^2\cos^2{\varphi}}}.
\end{align}
Therefore,
\begin{equation}
\dot{\varphi}=\frac{(a^2\sin^2{\varphi}+b^2\cos^2{\varphi})^{3/2}}{a^2b^2}.
\label{s2t}
\end{equation}
\subsection{Experiment}
We studied the trajectory of a pursuer starting from the origin as the evader moved once around the ellipse from $(a,0)$. We set $n=0.5$, $a=1.0$ and $b=0.5$. For the case in equation \eqref{param-1}, we  obtained numerical solutions $(x_1(t), y_1(t))$ of equations \eqref{eq-x} and \eqref{eq-y} from $t=0$ to $t=2\pi$. For the case in equation \eqref{param-2}, we calculated $(x_2(t), y_2(t))$ of the same equations for the same interval. For the case in equation \eqref{param-3}, we computed $(x_3(\varphi), y_3(\varphi))$ of equations \eqref{eq-x-phi} and \eqref{eq-y-phi} from $\varphi=\pi/2$ to $\varphi=\pi/2+2\pi$.\\ 
The results are shown in Fig.~\ref{consistent}. As demonstrated, the pursuer's trajectories are consistent across all cases.
Fig.~\ref{xzahyou} plots the x-coordinates of the evaders in Fig.~\ref{consistent} (A), (B), and (C). To facilitate comparison, we have set the horizontal axes of panels Fig.~\ref{consistent} (A) and (B) to $t$, while in panel Fig.~\ref{consistent} (C), the horizontal axis is set to $\varphi-\pi/2$. The position of three evaders coincide at $t=\varphi-\pi/2=n\pi/2$, ($n\in\mathbb{Z}$), and at this moment, the pursuer's x-coordinates also coincide.
\begin{figure}[htbp]
 \begin{subfigure}{0.32\linewidth}
  \centering
  \includegraphics[width=\linewidth]{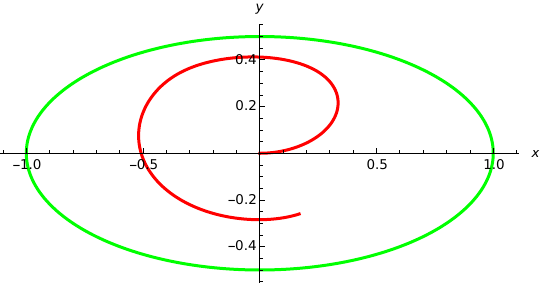}
  \caption{Red:Pursuer$(x_1(t), y_1(t))$,\\Green:Evader$(X_1(t), Y_1(t))$.}
 \end{subfigure}
 \begin{subfigure}{0.32\linewidth}
  \centering
  \includegraphics[width=\linewidth]{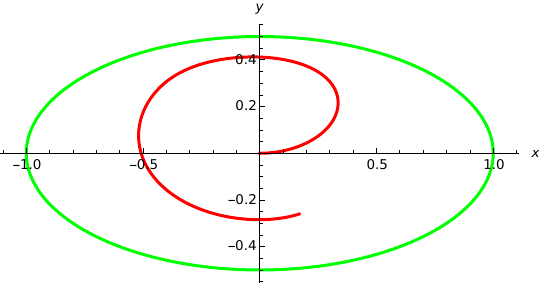}
  \caption{Red:Pursuer$(x_2(t),y_2(t))$,\\Green:Evader$(X_2(t), Y_2(t))$.}
  \end{subfigure}
  \begin{subfigure}{0.32\linewidth}
  \centering
  \includegraphics[width=\linewidth]{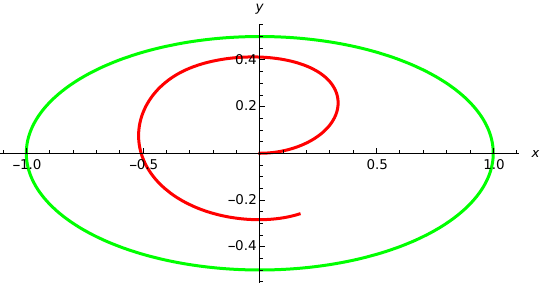}
  \caption{Red:Pursuer$(x_3(\varphi),y_3(\varphi))$,\\Green:Evader$(X_3(\varphi), Y_3(\varphi))$.}
 \end{subfigure}
  \caption{Pursuer's trajectory comparison among parameters \eqref{param-1}, \eqref{param-2}, and \eqref{param-3}. }
  \label{consistent}
\end{figure}
\begin{figure}[htbp]
\centering
\includegraphics[width=\linewidth]{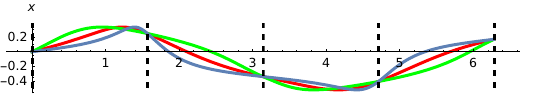}
\caption{Red: $(t,x_1(t))$, Green: $(t,x_2(t))$, Blue: $(\varphi-\pi/2,x_3(\varphi))$,\newline 
Dashed line: $t=\varphi-\pi/2=0,\pi/2,\pi,3\pi/2,2\pi$}
\label{xzahyou}
\end{figure}
\section{A New Theorem on Parametrization}
We have derived a theorem stating that transforming the evader’s parameters does not change the shape of the pursuer’s trajectory. In other words, the pursuer’s trajectory is unique, and if the evader’s parameters are changed, only the parameters of that trajectory are altered accordingly.

\begin{theorem}
 Assume that $x_s(t), y_s(t), \lambda_s(t)$ satisfies \eqref{barton-tsuiseki1}-\eqref{barton-tsuiseki3} for a given $X(t), Y(t)$, $n$. If the evader's parameters are changed from $t$ to $u$ while maintaining the same orientation, then $x_s(u), y_s(u)$ and $\lambda_s(u)$ satisfy the following system of differential equations, where $t$ in \eqref{barton-tsuiseki1}-\eqref{barton-tsuiseki3} is replaced by $u$.

\begin{empheq}[left={\empheqlbrace}]{align}
&X(u)=x(u)+\lambda(u)\frac{dx(u)}{du}, \label{barton-tsuiseki1-u}\\
&Y(u)=y(u)+\lambda(u)\frac{dy(u)}{du}, \label{barton-tsuiseki2-u}\\
&n^2\left(\left(\frac{dX(u)}{du}\right)^2+\left(\frac{dY(u)}{du}\right)^2\right)=\left(\frac{dx(u)}{du}\right)^2+\left(\frac{dy(u)}{du}\right)^2.\label{barton-tsuiseki3-u}
\end{empheq}\label{theorem1}
\end{theorem}
\begin{proof}
We derive a formula relating $\lambda_s(t)$ and $\lambda_s(u)$. From the definition of $\lambda_s(t)$, the following equation holds. 
\begin{align}
\lambda_s(t)&=\frac{\sqrt{(X(t)-x_s(t))^2+(Y(t)-y_s(t))^2}}{\sqrt{\left(\frac{dx_s(t)}{dt}\right)^2+\left(\frac{dy_s(t)}{dt}\right)^2}}.
\end{align}
Since $u$ is a function of $t$, the following two equations hold: 
\begin{align}
\frac{dx_s(t)}{dt}=\frac{dx_s(u(t))}{dt}&=\frac{dx_s(u)}{du}\frac{du}{dt}, \\
\frac{dy_s(t)}{dt}=\frac{dy_s(u(t))}{dt}&=\frac{dy_s(u)}{du}\frac{du}{dt}.
\end{align}
Moreover, since $X(u)=X(t)$, $Y(u)=Y(t)$, $x_s(u)=x_s(t)$ and $x_s(u)=x_s(t)$, the function $\lambda_s(t)$ is transformed as follows: 
\begin{align}
\lambda_s(t)&=\frac{\sqrt{(X(u)-x_s(u))^2+(Y(u)-y_s(u))^2}}{\frac{du}{dt}\sqrt{\left(\frac{dx_s(u)}{du}\right)^2+\left(\frac{dy_s(u)}{du}\right)^2}}=\frac{\lambda_s(u)}{\frac{du}{dt}}.
\end{align}
Therefore, $\lambda_s(u)=\lambda_s(t)\frac{du}{dt}$. By applying the chain rule, we have: 
\begin{align}
x_s(u)+\lambda_s(u)\frac{dx_s(u)}{du}&=x_s(t)+\lambda_s(t)\frac{du}{dt}\frac{dx_s(t(u))}{dt}\frac{dt}{du}\\
&=x_s(t)+\lambda_s(t)\frac{dx_s(t)}{dt}
\intertext{Since $x_s(t)$, $\lambda_s(t)$ are solutions of \eqref{barton-tsuiseki1}, we have:}
&=X(t).
\end{align}
As the parameter $u$ is transformed from $t$, which is the parameter of the evader's trajectory, $X(t)=X(u(t))$. Therefore, $x_s(u)+\lambda_s(u)\frac{dx_s(u)}{du}=X(u)$. 

We can prove the equation for $y_s(u)+\lambda_s(u)\frac{dy_s(u)}{du}=Y(u)$ in the same manner. Also, The equation \eqref{barton-tsuiseki3-u} for the speed ratio can also be shown as follows:
\begin{align}
\left(\frac{dx_s(u)}{du}\right)^2+\left(\frac{dy_s(u)}{du}\right)^2&=\left(\left(\frac{dx_s(t)}{dt}\right)^2+\left(\frac{dy_s(t)}{dt}\right)^2\right)\left(\frac{dt}{du}\right)^2\\
&=n^2\left(\left(\frac{dX(t)}{dt}\right)^2+\left(\frac{dY(t)}{dt}\right)^2\right)\left(\frac{dt}{du}\right)^2\\
&=n^2\left(\left(\frac{dX(u)}{du}\right)^2+\left(\frac{dY(u)}{du}\right)^2\right).
\end{align}
\end{proof}

The same argument applies to equations \eqref{eq-x} and \eqref{eq-y}, which were used in Section 3 to numerically calculate the trajectory of the pursuer.

\begin{theorem}
 Assume that $x_s(t)$ and $y_s(t)$ satisfy equations \eqref{eq-x} and \eqref{eq-y} for a given $X(t),Y(t)$ and $n$. If the evader’s parameters are changed from 
$t$ to $u$ while maintaining the orientation, then $x_s(u)$ and $y_s(u)$ satisfy the following system of differential equations, where $t$ in equations \eqref{eq-x} and \eqref{eq-y} is replaced by $u$:

\begin{empheq}[left={\empheqlbrace}]{align}
\frac{dx(u)}{du}&=\sqrt{(\frac{dX(u)}{du})^2+(\frac{dY(u)}{du})^2}\frac{n(X(u)-x(u))}{\sqrt{(X(u)-x(u)^2+(Y(u)-y(u))^2}},\label{eq-x-u}\\
\frac{dy(u)}{du}&=\sqrt{(\frac{dX(u)}{du})^2+(\frac{dY(u)}{du})^2}\frac{n(Y(u)-y(u))}{\sqrt{(X(u)-x(u))^2+(Y(u)-y(u))^2}}.\label{eq-y-u}
\end{empheq}\label{theorem}
\end{theorem}
\begin{proof}
Since $u$ is a function of $t$, and $t$ is a function of $u$, we have:
\begin{align}
\frac{dX(u)}{du}&=\frac{dX(t(u))}{dt}\frac{dt}{du}.\label{eq-5.1}
\end{align}
The same holds for $Y$, $x_s$ and $y_s$. Therefore,
\begin{equation}
\sqrt{(\frac{dX(u)}{du})^2+(\frac{dY(u)}{du})^2}=\frac{dt}{du}\sqrt{\dot{X}^2+\dot{Y}^2}.\label{eq-5.2}
\end{equation}
Since $X(u)=X(t)$, $Y(u)=Y(t)$, $x_s(u)=x_s(t)$, and $x_s(u)=x_s(t)$, the right-hand side of equation \eqref{eq-x-u} transforms as follows:
\begin{align}
&n\sqrt{(\frac{dX(u)}{du})^2+(\frac{dY(u)}{du})^2}\frac{X(u)-x_s(u)}{\sqrt{(X(u)-x_s(u))^2+(Y(u)-y(u))^2}},\\
&=\frac{dt}{du}\sqrt{\dot{X}^2+\dot{Y}^2}\frac{n(X(t)-x_s(t))}{\sqrt{(X(t)-x_s(t))^2+(Y(t)-y_s(t))^2}}.
\end{align}
Substituting \eqref{eq-5.1} and \eqref{eq-5.2} for this equation, we have:
\begin{align}
\frac{dt}{du}\sqrt{\dot{X}^2+\dot{Y}^2}\frac{n(X(t)-x_s(t))}{\sqrt{(X(t)-x_s(t))^2+(Y(t)-y_s(t))^2}}=\frac{dt}{du}\frac{dx_s(t)}{dt}=\frac{dx_s(u)}{du}.
\end{align}
The right side of equation \eqref{eq-y-u} can be transformed similarly, leading to:
\begin{equation}
\sqrt{(\frac{dX(u)}{du})^2+(\frac{dY(u)}{du})^2}\frac{n(Y(u)-y_s(u))}{\sqrt{(X(u)-x_s(u))^2+(Y(u)-y(u))^2}}=\frac{dy_s(u)}{du}.
\end{equation}
\end{proof}

\section{Derivation of a Dynamical System}
Theorem.1 allows us to assume the evader's speed is consistently $1$ when we are only concerned with the shape of the pursuer's trajectory. From this assumption, a dynamical system can be derived from this assumption. First, let $|\dot{\mathbf{E}}|=1$. From equation \eqref{eq-speed}, $|\dot{\mathbf{P}}|=n$. As in Section 3.2, we parameterize the velocity vector as follows:
\begin{align}
\dot{X}=\cos{\Theta(t)},\quad \dot{Y}=\sin{\Theta(t)}.\label{DDS-1}\\
\dot{x}=n\cos{\theta(t)},\quad \dot{y}=n\sin{\theta(t)}.\label{DDS-2}
\end{align}
Second, let $\rho(t)=|\mathbf{E}-\mathbf{P}|$. Taking the absolute value in \eqref{eq-muki},
\begin{align}
|\lambda| |\dot{\mathbf{P}}|&=|\mathbf{E}-\mathbf{P}|\\
\lambda n&=\rho(t)\\
\lambda(t)&=\frac{\rho(t)}{n}.\label{rho-lambda}
\end{align}
Substituting equation \eqref{rho-lambda} for \eqref{barton-tsuiseki1} and \eqref{barton-tsuiseki2},
\begin{equation}
\left \{
   \begin{aligned}
X&=x+\frac{\rho}{n}\dot{x},\\
Y&=y+\frac{\rho}{n}\dot{y}.
\end{aligned} \label{DDS-3}
\right.
\end{equation}
Third, we derive a simultaneous differential equation for $\Theta(t)$ and $\theta(t)$. The following equation is the result of differentiating equation \eqref{DDS-3} with respect to $t$, written for each component.
\begin{equation}
\left \{
   \begin{aligned}
\dot{X}&=(1+\frac{\dot{\rho}}{n})\dot{x}+\frac{\rho}{n}\ddot{x},\\
\dot{Y}&=(1+\frac{\dot{\rho}}{n})\dot{y}+\frac{\rho}{n}\ddot{y}.
\end{aligned} \label{DDS-4}
\right.
\end{equation}
Substituting equations \eqref{DDS-1} and \eqref{DDS-2} into equation \eqref{DDS-4} and applying the addition theorem of trigonometric functions, we obtain:
\begin{equation}
\left\{
\begin{aligned}
\dot{\rho}&=\cos(\Theta-\theta)-n, \\
\rho\dot{\theta}&=\sin(\Theta-\theta).
\end{aligned}\label{DDS-5}
\right.
\end{equation}
Finally, we introduce a new variable $\zeta \coloneqq \Theta-\theta$. \eqref{DDS-5} is transformed into the following more analyzable simultaneous differential equation:
\begin{equation}
\left\{
\begin{aligned}
\dot{\rho}&=\cos \zeta-n, \\
\rho\dot{\zeta}&=-\sin \zeta+\rho\dot{\Theta}.
\end{aligned}\label{DDS-6}
\right.
\end{equation}
The simultaneous differential equation \eqref{DDS-6} represents a dynamical system. It is simpler than the original simultaneous differential equation \eqref{barton-tsuiseki1}-\eqref{barton-tsuiseki3}.
\section{Dynamical System in Circular Pursuit and Evasion}
We discuss the case of circular pursuit and evasion. The parameters for a circle $X^2+Y^2=1$ are $X(t)=a\cos t$ and $Y(t)=a\sin t $. These are not consistent with $|\dot{\mathbf{E}}|=1$, so we modify $t$ to $t/a$. Therefore, the parameters change to $X(t)=a\cos \frac{t}{a}$ and $Y(t)=a\sin \frac{t}{a}$. By differentiating, we find that $\Theta(t)=t/a+\pi/2$ and $\dot{\Theta}=\frac{1}{a}$. Thus, the dynamical system \eqref{DDS-6} for circular pursuit and evasion becomes:
\begin{equation}
\left\{
\begin{aligned}
\dot{\rho}&=\cos \zeta-n, \\
\rho\dot{\zeta}&=-\sin \zeta+\frac{\rho}{a}.\label{DDS-7}
\end{aligned}
\right.
\end{equation}
If $n<1$, \eqref{DDS-7} has an equilibrium point $(n, \rho)=(\cos \zeta, a\sin \zeta)\Leftrightarrow (\rho, \zeta)=(a\sqrt{1-n^2}, \cos^{-1}n) $. Therefore, the pursuer's trajectory converges to a reduced circle scaled by a factor of $n$.

\subsection{Stability in Equilibrium Point}
In this subsection, we demonstrate that the equilibrium point \eqref{DDS-7}, $(n, \rho^*)=(\cos \zeta^*, a\sin \zeta^*)\Leftrightarrow (\rho^*, \zeta^*)=(a\sqrt{1-n^2}, \cos^{-1}n) $ is asymptotically stable. Proving that this equilibrium point is globally asymptotically stable is more challenging. Let:
\begin{align}
    f(\rho, \zeta)=\cos \zeta-n,\\
    g(\rho, \zeta)=-\frac{\sin \zeta}{\rho}+a.
\end{align}
If $(\rho, \zeta)$ is close to $(\rho^*, \zeta^*)$, the stability of \eqref{DDS-7} is determined by the eigenvalues of the following matrix:
\begin{align}
    \begin{pmatrix}
     \dfrac{\partial f}{\partial \rho}(\rho^*, \zeta^*) &  \dfrac{\partial f}{\partial \zeta}(\rho^*, \zeta^*)  \\
     \dfrac{\partial g}{\partial \rho}(\rho^*, \zeta^*) &  \dfrac{\partial g}{\partial \zeta}(\rho^*, \zeta^*)
    \end{pmatrix}
    &=
    \begin{pmatrix}
    0 &  -\sin \zeta^{*}  \\
     \dfrac{\sin \zeta^{*}}{(\rho^*)^2} &  -\dfrac{\cos \zeta^*}{\rho^*}    
    \end{pmatrix}\\
    &=\begin{pmatrix}
    0 &  -\sqrt{1-n^2} \\
     \dfrac{1}{a^2\sqrt{1-n^2}} &  -\dfrac{n}{a\sqrt{1-n^2}}    
    \end{pmatrix}
\intertext{Let matrix $A$ be defined as follows:}
A&=\begin{pmatrix}
    0 &  -\sqrt{1-n^2} \\
     \dfrac{1}{a^2\sqrt{1-n^2}} &  -\dfrac{n}{a\sqrt{1-n^2}}.    
    \end{pmatrix}
\end{align}
Matrix $A$ has two eigenvalues, which are:
\begin{align}
    \lambda_{\pm}=-\frac{n}{2a\sqrt{1-n^2}}\pm \frac{\sqrt{5n^2-4}}{2a\sqrt{1-n^2}}
\end{align}
If $0<n<2/\sqrt{5}$, $5n^2-4$ is negative, so $\lambda_{\pm}$ are complex numbers with negative real parts. If $2/\sqrt{5}<n<1$, $\sqrt{5n^2-4}$ less than $n$, so $\lambda_{\pm}$ are real negative numbers. Therefore, $(\rho^*, \zeta^*)$ is asymptotically stable.

\section{Dynamical System in Elliptical Pursuit and Evasion}
We examine the elliptical pursuit and evasion scenario. A function $\Theta$ in equation \eqref{DDS-6} corresponds to  $\varphi$ in Section 3.2. Substituting equation \eqref{s2t} into equation \eqref{DDS-6}, we obtain:
\begin{equation}
\left\{
\begin{aligned}
\dot{\rho}&=\cos \zeta-n, \\
\rho\dot{\zeta}&=-\sin \zeta+\rho(a^2\sin^2(\varphi)+b^2\cos^2(\varphi))^{3/2}/(a^2b^2).
\end{aligned}\label{DDS-8}
\right.
\end{equation}
In the case of circular pursuit and evasion, equation \eqref{DDS-6} becomes equation \eqref{DDS-7} which is an autonomous dynamical system. However, equation \eqref{DDS-6} becomes equation \eqref{DDS-8} which is non-autonomous when considering elliptical pursuit and evasion. Since equation \eqref{DDS-8} lacks information about $\dot{\varphi}$, both equations \eqref{DDS-8} and \eqref{s2t} must be solved simultaneously to obtain numerical solutions. Note that if $a=b$, equation \eqref{DDS-8} coincides with equation \eqref{DDS-7}, and so it reduces to the circular pursuit and evasion case.

We illustrate the difference between these two dynamical systems by drawing two $\rho-\zeta$, as shown in Fig.\eqref{circle-ellipse}. Fig.\eqref{enn-n<1} represents the phase portrait of equation \eqref{DDS-7} with parameters $n=0.5, a=1.0$ and $ b=1.0$. We obtained a numerical solution of equation \eqref{DDS-7} from $t=0$ to $t=10\pi$, with initial conditions $\rho(0)=1.0$ and $\zeta(0)=\pi/2$. In this figure, the solution trajectory terminates at an equilibrium point with coordinates $(\cos^{-1}0.5,\sqrt{1-0.5^2})$. Fig.\eqref{enn-n>1} represents the phase portrait of equation \eqref{DDS-8}, where $b$ is set to $0.5$ and the other parameters are the same as in Fig. \eqref{enn-n<1}. We obtained a numerical solution of equation \eqref{DDS-8} from $t=0$ to $t=10\pi$, with initial conditions $\rho(0)=1.0$ and $\zeta(0)=\pi/2$. As shown in the figure, the solution trajectory converges to a closed curve, although no mathematical proof has been established for this observation. Additionally, the shape of this closed curve remains unclear.
\begin{figure}[htbp]
 \centering
 \begin{subfigure}{0.4\columnwidth}
  \centering
  \includegraphics[width=0.50\columnwidth]{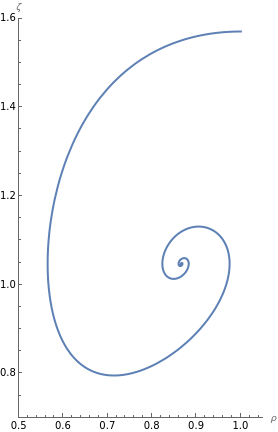}
  \caption{Circle, a=b=1.0
  }
  \label{enn-n<1}
 \end{subfigure}
  \begin{subfigure}{0.43\columnwidth}
  \centering
  \includegraphics[width=0.35\columnwidth]{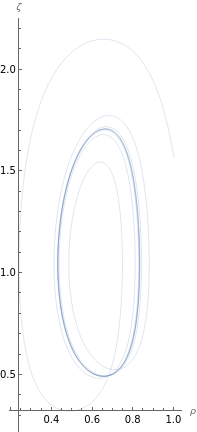}
  \caption{Ellipse, a=1.0, b=0.5
  }
  \label{enn-n>1}
 \end{subfigure}
 \caption{Dynamical system difference between circular and elliptical chase and escape. The vertical line shows $\rho$ and the horizontal one $\zeta$. Numerical solutions from $t=0$ to $t=10\pi$, with initial conditions $\rho(0)=1.0$ and $\zeta(0)=\pi/2$. We set $n=0.5$.
 }
 \label{circle-ellipse}
\end{figure}

\section{Equation about $\zeta(\Theta)$}
In this section we derive a single differential equation from \eqref{DDS-6}. Assume that the argument of the evader's velocity vector, $\Theta(t)$, is strictly monotonically increasing. This assumption holds when the Evader is orbiting a circle or an ellipse in a counterclockwise direction. Under this assumption, the time variable $t$ can be converted into $\Theta$. Equation \eqref{DDS-6} then transforms as follows:
\begin{empheq}[left={\empheqlbrace}]{align}
\rho^{\prime}&=f(\Theta)(\cos \zeta-n), \label{rho-Theta-1}\\
\rho\zeta^{\prime}&=-f(\Theta)\sin \zeta+\rho.\label{rho-Theta-2}
\end{empheq}
Where $\prime$ denotes the derivative with respect to $\Theta$, and $f(\Theta)\coloneqq\dfrac{1}{\dot{\Theta}}$. Differentiating the second equation \eqref{rho-Theta-2} with respect to $\Theta$, we obtain:
\begin{align}
    -\zeta^{\prime\prime}\rho+(1-\zeta^{\prime})\rho^{\prime}=f^{\prime}\sin \zeta+f\zeta^{\prime}\cos\zeta.
\end{align}
Substitute the first equation \eqref{rho-Theta-1} for $\rho^{\prime}$ and multiplying both sides by $(1-\zeta^{\prime})$:
\begin{align}
     -\zeta^{\prime\prime}\rho(1-\zeta^{\prime})+(1-\zeta^{\prime})^2f(\cos \zeta-n)=(f^{\prime}\sin \zeta+f\zeta^{\prime}\cos\zeta)(1-\zeta^{\prime}).
\end{align}
From the second equation \eqref{rho-Theta-2}, $\rho(1-\zeta^{\prime})=f\sin\zeta$, therefore we obtain the following second-order nonlinear differential equation for $\zeta(\Theta)$.
\begin{align}
-\zeta^{\prime\prime}f\sin\zeta+(1-\zeta^{\prime})^2(\cos\zeta-n)f=\left(f^{\prime}\sin\zeta+f\zeta^{\prime}\cos\zeta\right)\left(1-\zeta^{\prime}\right)\label{zeta-varphi}
\end{align}
We now derive \eqref{zeta-varphi} for the elliptical pursuit and evasion case. Divide equation \eqref{zeta-varphi} by $f$,
\begin{align}
-\zeta^{\prime\prime}\sin\varphi+(1-\zeta^{\prime})^2(\cos\zeta-n)=\left(\frac{f^{\prime}}{f}\sin\zeta+\zeta^{\prime}\cos\zeta\right)\left(1-\zeta^{\prime}\right).
\end{align}
Substituting equation \eqref{s2t} for $\frac{f^{\prime}}{f}=(-\frac{1}{f})^{\prime}f$, we have:
\begin{align}
    \frac{f^{\prime}}{f}=\frac{-3(a^2-b^2)\sin\varphi\cos\varphi}{a^2\sin^2\varphi+b^2\cos^2\varphi}
\end{align}
Thus, the dynamical system in elliptical pursuit and evasion, given by  equation \eqref{DDS-8}, can be rewritten as the following two-order nonlinear differential equation:
\begin{align}
-\zeta^{\prime\prime}\sin\zeta+(1-\zeta^{\prime})^2(\cos\zeta-n)=\left(\frac{-3(a^2-b^2)\sin\varphi\cos\varphi}{a^2\sin^2\varphi+b^2\cos^2\varphi}\sin\zeta+\zeta^{\prime}\cos\zeta\right)\left(1-\zeta^{\prime}\right).
\end{align}
\section{Conclusion}
This paper explored two key findings. First, we demonstrated that changing the parameterization of the evader does not alter the shape of the pursuer’s trajectory. As a result, we can infer that the fact that the pursuer’s trajectory does not become an ellipse scaled by a factor of $n$ is not solely due to the evader’s speed or acceleration, but rather the specific shape of the ellipse itself. Second, we examined the pursuit-evasion game from the perspective of dynamical systems. The main difference between circular and elliptical pursuit and evasion lies in whether the dynamical system is autonomous or non-autonomous, and in the shape of the solution trajectories. 
\section{Future Works}
There are two important areas for future work. First, we need to show that the dynamical system for circular pursuit and evasion, described by equation \eqref{DDS-7}, has a globally asymptotically stable equilibrium point. Second, it is crucial to demonstrate that the solution trajectory of the dynamical system for elliptical pursuit and evasion, given by equation \eqref{DDS-8}, converges to a closed curve. Additionally, solving the differential equation governing the angular difference between the two players, as described by equation \eqref{zeta-varphi}, is an important problem that remains to be addressed.
\section{Acknowledgement}
This paper was completed under the guidance of professor Toru Ohira of the Graduate School of Mathematics, Nagoya University. This work was financially supported by JST SPRING, Grant Number JPMJSP2125. The author S.Y. would like to take this opportunity to thank the “THERS Make New Standards Program for the Next Generation Researchers.”

\end{document}